\documentclass[]{amsart}
\usepackage[utf8]{inputenc}
\usepackage{graphicx}

\DeclareUnicodeCharacter{2212}{-}

\usepackage{mathrsfs}
\usepackage{booktabs}
\usepackage{graphicx}

\usepackage{amsmath, amssymb}
\usepackage{amsthm}
\usepackage{enumerate}
\usepackage{mathtools}
\usepackage{tikz-cd}
\usepackage{comment}
\usepackage{color}
\usepackage{hyperref}
\usepackage{soul}
\usepackage{enumitem}
\usepackage{scalerel}[2014/03/10]
\usepackage{stackengine}

\newcommand{\CAT}{\mathrm{CAT}}
\newcommand{\GL}{\operatorname{GL}}

\newcommand{\p}{\operatorname{\mathbb{P}}}

\renewcommand{\P}{\operatorname{\mathbb{P}}}

\newcommand{\Z}{\operatorname{\mathbb{Z}}}
\newcommand{\h}{\operatorname{\mathbb{H}^{\infty}}}

\newcommand{\N}{\operatorname{\mathbb{Z}_{\geq 0}}}

\newcommand{\R}{\operatorname{\mathbb{R}}}
\newcommand{\F}{\operatorname{\mathbb{F}}}
\newcommand{\kk}{k}

\newcommand{\bp}{\mathfrak{b}}
\newcommand{\id}{\operatorname{id}}

\newcommand{\Aut}{\operatorname{Aut}}

\newcommand{\Bir}{\operatorname{Bir}}
\newcommand{\Cr}{\operatorname{Cr}}
\newcommand{\PGL}{\operatorname{PGL}}

\newcommand{\Jon}{\mathscr{J}}

\newcommand{\B}{\mathcal{B}}

\def\dashmapsto{\mapstochar\dashrightarrow}

\def\dashmapsto{\mapstochar\dashrightarrow}

\newcommand{\CC}{\mathcal{C}}

\theoremstyle{plain}
\newtheorem{theorem}{Theorem}[section]
\newtheorem{thm}[theorem]{Theorem}
\newtheorem{lemma}[theorem]{Lemma}
\newtheorem{proposition}[theorem]{Proposition}
\newtheorem{corollary}[theorem]{Corollary}

\theoremstyle{definition}

\newtheorem{defin}[theorem]{Definition}

\newtheorem{remark}[theorem]{Remark}
\newtheorem{rem}[theorem]{Remark}
\newtheorem{question}[theorem]{Question}
\newtheorem{conv}[theorem]{Convention}

\title[Finitely generated subgroups of algebraic elements]{Finitely generated subgroups of algebraic elements of plane Cremona groups are bounded} 
\author[A.~Lonjou]{Anne Lonjou$^{\ast}$}

	\address{
	Department of Mathematics, University of the
	Basque Country UPV/EHU,  Sarriena s/n, 48940 Leioa, Bizkaia, Spain. 
	IKERBASQUE, Basque Foundation for Science, Bilbao, Spain.}
	\email{anne.lonjou@ehu.eus}
\author[P.~Przytycki]{Piotr Przytycki$^{\dag}$}
	
	\address{
		Department of Mathematics and Statistics,
		McGill University,
		Burnside Hall,
		805 Sherbrooke Street West,
		Montreal, QC,
		H3A 0B9, Canada}
	
	\email{piotr.przytycki@mcgill.ca}
	
		\address{Department of Mathematics, 
			ETH Z\"urich,
			Rämistrasse 101, 
			8092 Z\"urich, Switzerland}
	
	\email{christian.urech@gmail.com}
	
	\thanks{$\ast$ Partially supported by MCIN /AEI /10.13039/501100011033 / FEDER through the Spanish Governement grant PID2022-138719NA-I00, by the french grant ANR-22-CE40-0004 GOFR and by the Basque Government grant IT1483-22.}
		
		\thanks{$\dag$ Partially supported by NSERC and (Polish) Narodowe Centrum Nauki, UMO-2018/30/M/ST1/00668.}

\author[C.~Urech]{Christian Urech}

\begin{document}
	\maketitle
	
\begin{abstract}
		\noindent
		We prove that any finitely generated subgroup of the plane Cremona group consisting only of algebraic elements is of bounded degree. This follows from a more general result on `decent' actions on infinite restricted products. We apply our results to describe the degree growth of finitely generated subgroups of the plane Cremona group. 
	\end{abstract}
			
\section{Introduction} 

The Cremona group $\Cr_2(k)$ over a field $k$ is the group of birational transformations of the projective plane $\p^2$ over $k$. Cremona groups have been the delight of algebraic geometers and group theorists in both, classical and modern times. As of today, many aspects of $\Cr_2(k)$ are well understood and there are many tools at hand to study those groups. For instance, $\Cr_2(k)$ acts by isometries on an infinite dimensional hyperbolic space $\mathbb{H}^\infty$ \cite{Cantat_groupes_birat} and on various CAT(0) cube complexes \cite{lonjou2021actions}. Nevertheless, some questions have remained open. The goal of this article is to positively answer a question asked more than a decade ago by Favre in \cite[Question~1]{favrebourbaki}.

Let us fix projective coordinates $[x:y:z]$ of $\p^2$. An element $f\in\Cr_2(k)$ is given by
\[
[x:y:z]\dashmapsto [f_0:f_1:f_2],
\] 
where the $f_i\in k[x,y,z]$ are homogeneous of the same degree and without non-constant common factor. The \emph{degree} $\deg(f)$ of $f$ is defined as the degree of the polynomials $f_i$. An element $f\in\Cr_2(k)$ is \emph{algebraic}, if $\deg(f^n)$ is uniformly bounded for all $n\in\Z$. A subgroup $G<\Cr_2(k)$ is \emph{bounded} if the degree of all elements in $G$ is uniformly bounded. 
Clearly, a bounded subgroup consists of algebraic elements. However, the converse is not true. For instance, consider the subgroup defined in the affine coordinates $(x,y)$  of $\p^2$ by 
\[
G=\{(x,y)\dashmapsto (x+R(y), y)\mid R\in k(y)\}.
\]
 Every element in $G$ is algebraic, but $G$ is not bounded. In this paper, we show the following theorem, which solves Favre's question: 

\begin{theorem}\label{thm:main}
	{Let $k$ be a field and} let $G<\Cr_2(k)$ be a finitely generated subgroup such that every element of $G$ is algebraic. Then $G$ is bounded.
\end{theorem}

\begin{remark}\label{rmk_alg_closed_fields}
Note that the properties of being algebraic and being bounded are invariant under field extensions, so it is enough to show Theorem~\ref{thm:main} for algebraically closed fields.  
\end{remark}

\begin{remark}\label{rmk_redu_Cantat}
The first step towards the proof of Theorem~\ref{thm:main} is due to Cantat, who showed in \cite{Cantat_groupes_birat} that a finitely generated subgroup $G<\Cr_2(k)$ consisting of algebraic elements is bounded or preserves a rational fibration (see \cite{Lamy_Cremona} for a proof of this result for fields of arbitrary characteristic). It is therefore enough to show Theorem~\ref{thm:main} for finitely generated groups that preserve a rational fibration. 
\end{remark}

{In order to show Theorem~\ref{thm:main} for finitely generated groups that preserve a rational fibration}, we introduce the \emph{Jonqui\`eres complex} --- a $\mathrm{CAT}(0)$ cube complex with an isometric action of the group of birational transformations preserving a rational fibration, whose vertex stabilisers are bounded subgroups. Using a suitable description of this complex and the dynamics of $\PGL_2(k)$ on $\p^1$, we show that the action of our group on this complex is decent, as defined below. 

 \begin{defin}
\label{def:decent}
Let $X_0$ be a set. We say that a group $G_0$ acts on $X_0$ \emph{purely elliptically}, if each element of $G_0$ fixes a point of $X_0$. We say that a group $G_0$ acts on~$X_0$ \emph{decently} if 
\begin{itemize}
\item
each subgroup of $G_0$ with a finite orbit fixes a point of $X_0$, and
\item
each finitely generated subgroup of $G_0$ acting purely elliptically fixes a point of $X_0$.
\end{itemize}
\end{defin}

It is an easy exercise 
that if $G_0$ is the isometry group of a simplicial tree $X_0$, then $G_0$ acts on $X_0$ decently {(see \cite[Corollary 3 in \S6.5 and Example 6.3.4]{Serre})}. More generally, if $G_0$ is the isometry group of a $\mathrm{CAT}(0)$ cube complex $X_0$ with no infinite cubes, then $G_0$ acts on $X_0$ decently \cite{GLU_Neretin}. Similarly, if $G_0$ is the isometry group of a $\mathrm{CAT}(0)$ $2$-complex $X_0$ with rational angles, then $G_0$ acts on~$X_0$ decently \cite{NOP}. For further examples, see \cite{HO}.

\subsection{Applications.} 
The Cremona group $\Cr_2(k)$ can be equipped with the \emph{Zariski topology} (see for instance \cite{Serre_Bourbaki}). An \emph{algebraic subgroup} of $\Cr_2(k)$ is a Zariski closed bounded subgroup. This explains the terminology: an element in $\Cr_2(k)$ is algebraic if and only if it is contained in an algebraic subgroup. Note that for any $d\geq1$, the subset of $\Cr_2(k)$ consisting of elements of degree at most $d$ is closed (see \cite{Blanc_Furter}), hence a bounded subgroup is contained in an algebraic subgroup. An algebraic subgroup $G$ is always \emph{projectively regularizable}, i.e., there exists a birational map $\varphi\colon \p^2\dashrightarrow S$ such that $\varphi G\varphi^{-1}<\Aut(S)$ for some regular projective surface $S$. This follows from the theorems of Weil and Sumihiro or from the fact that the number of base-points of elements in an algebraic subgroup is uniformly bounded (we refer to \cite{Lamy_Cremona} for references and a proof of this fact, or to \cite{lonjou2021actions}). Theorem~\ref{thm:main} has therefore the following direct consequence:

\begin{corollary}
\label{cor:main}
		{Let $k$ be a field and}  let $G<\Cr_2(k)$ be a finitely generated subgroup such that every element of $G$ is algebraic. Then $G$ is projectively regularizable. \end{corollary}

From another point of view, algebraic elements correspond exactly to elements in~$\Cr_2(k)$ inducing an elliptic isometry on the infinite dimensional hyperbolic space~$\mathbb{H}^\infty$, on which $\Cr_2(k)$ acts \cite{Cantat_groupes_birat}. Theorem~\ref{thm:main} therefore states that the action of $\Cr_2(k)$ on~$\mathbb{H}^\infty$ is decent. 

While all algebraic elements are projectively regularizable, there exist also non-algebraic elements that are projectively regularizable. It is still unknown whether a finitely generated subgroup of the Cremona group containing only projectively regularizable elements is projectively regularizable or not. This question is equivalent to the question whether $\Cr_2(k)$ acts decently on the \emph{blow-up complex} --- a CAT(0) cube complex constructed in \cite{lonjou2021actions}. In \cite{GLU_Neretin}, the first and third authors together with Genevois positively answer this question when the base-field {$k$} is finite.

Since elements in $\Cr_2(k)$ preserving a rational fibration are projectively regularizable if and only if they are algebraic {(see for instance \cite[Table 1]{lonjou2021actions})}, Corollary~\ref{cor:main} directly implies the following.
	Let $G<\Cr_2(k)$ be a finitely generated subgroup preserving a rational fibration such that every element in $G$ is projectively regularizable. Then $G$ is projectively regularizable.

The notions of algebraic elements {and bounded subgroups} generalize to birational transformations of arbitrary regular projective surfaces and Theorem~\ref{thm:main} generalizes to this setting:
\begin{theorem}\label{thm:main_general}
	Let $S$ be a regular projective surface over a field $k$ and let $G <\Bir(S)$ be a finitely generated subgroup such that every element of~$G$ is algebraic. Then $G$ is {bounded}.
\end{theorem}However, the most interesting and difficult case is the one of rational surfaces. In order to keep the notation more accessible, we discuss and prove the general case in the separate Section~\ref{sec:remarks}. 

In Section~\ref{sec:degrees} we apply Theorem~\ref{thm:main} to give a first description of the asymptotic degree growth of finitely generated subgroups of $\Cr_2(k)$. This opens up new interesting questions about the dynamical behaviour of finitely generated subgroups of~$\Cr_2(k)$. Let $G<\Cr_2(k)$ be a finitely generated subgroup with a finite generating set $T$. Denote by $B_T(n)$ the set of all elements in $G$ of word length $l_T$ at most $n$ with respect to the generating set $T$, and define

\[
D_{T}(n):=\max_{f\in B_T(n)}\{\deg(f)\}.
\]
It has been shown in \cite{MR3818624} that there are only countably many integer sequences that can appear in this way. However, still very little is known about them. In the case where $T$ consists of a single element, the growth of the function $D_T$ has been {extensively} studied (see Theorem~\ref{thm:hyperboloid}). Theorem~\ref{thm:main} is the main ingredient for the following result, which we prove in Section~\ref{sec:degrees}. For two functions, $f$ and $g$ on $\N$, we write $f\asymp g$ if $f(x)\leq ag(bx)$ and $g(x)\leq cf(dx)$ for some $a,b,c,d>0$.
\begin{corollary}\label{cor:degree}
	Let $\kk$ be an algebraically closed field and let $G<\Cr_2(k)$ be a finitely generated subgroup with generating set $T$. Then one of the following is true:
	\begin{enumerate}
		\item All elements in $G$ are algebraic, $G$ fixes a point in $\h$, and $D_T(n)$ is bounded.	
		\item[(2a)] The group $G$ contains an element that induces a parabolic isometry of $\h$, $D_T(n)\asymp n$, and $G$ preserves a rational fibration. 
		\item[(2b)] The group $G$ contains an element that induces a parabolic isometry of $\h$, $D_T(n)\asymp n^2$, and $G$ preserves an elliptic fibration. 
		\item[(3)] The group $G$ contains an element that induces a loxodromic isometry of $\h$ and  $D_T(n)\asymp\lambda^n$ for $\lambda>0$. 
	\end{enumerate}
\end{corollary}

It would be interesting to study the dynamical behaviour of the degrees of finitely generated subgroups in more detail. In Corollary~\ref{cor:degree}, the asymptotic behaviour does not depend on the choice of the finite generating set $T$, {since for any other generating set $T'$ there is a constant $d>0$ such that for all $n$ we have $B_{T'}(n)\subseteq B_T(dn)$, and so $D_{T'}(n)\leq D_{T}(dn)$.} For instance, one could ask, after fixing $T$, what is the precise asymptotic growth of $D_T(n)$, if the group $G$ contains a loxodromic element:

\begin{question}
	Let $G<\Cr_2(k)$ be a finitely generated subgroup containing an element that induces a loxodromic isometry of $\h$, with finite generating set $T$. What is the asymptotic growth of $D_T(n)$? Do we always have $D_T(n)\sim c\lambda^n$ for some constants $c,\lambda>0$? Here, we write $f\sim g$ if $f$ and $g$ are functions such that $\lim_{n\to\infty} \frac{f(n)}{g(n)}=1$.
	
\end{question}

In another direction, it has been shown that the \emph{dynamical degree} of $f\in\Cr_2(k)$, i.e., the number $\lambda(f):=\lim_{n\to\infty}\deg(f)^{1/n}$ is always an algebraic integer (more precisely, it is always a Pisot or Salem number, see \cite{blanc2016dynamical}). This leads to the following natural question, which is due to Cantat:

\begin{question}
	Let $G<\Cr_2(k)$ be a finitely generated subgroup with finite generating set $T$. Which real numbers can be realized as $\lim_{n\to\infty}D_T(G)^{1/n}$?
\end{question}

For instance, it could be interesting to construct examples of such subgroups  and generating sets such that $\lim_{n\to\infty}D_T(G)^{1/n}$ is transcendental. If $S$ is a {regular} projective surface and $f\in\Aut(S)$, then $\lambda(f)$ is the spectral radius of the induced transformation $f_*$ of $f$ on the Neron--Severi lattice of $S$. For this reason, the limit $\lim_{n\to\infty}D_T(G)^{1/n}$ should be seen as an analogue to the \emph{joint spectral radius}, which has been introduced in \cite{rota1960note} and has been studied by many authors (see for instance \cite{jungers2009joint} or \cite{breuillard2021joint}).

\subsection{Decent actions.}
In Section~\ref{sec:Jon} we will show that Theorem~\ref{thm:main} is a special case of a following more general result on groups acting decently on an infinite restricted product. 
 
\label{sum}
A \emph{pointed set} $(X_0,x_0)$ is a set $X_0$ and a point $x_0\in X_0$. The \emph{restricted product} $\bigoplus_{p\in P} (X_p,x_p)$ of a family $\{(X_p,x_p)\}_{p\in P}$ of pointed sets is the set of sections $\{y_p\}_{p\in P}$ with $y_p\in X_p$ such that all but finitely many $y_p$ are equal to $x_p$. 

Note that for finite $P$ we have $\bigoplus_{p\in P} (X_p,x_p)=\Pi_{p\in P} X_p$. For infinite $P$, we have that
$\bigoplus_{p\in P} (X_p,x_p)$ is a proper subset of $\Pi_{p\in P} X_p$.

If each $X_p$ is a simplicial tree and each $x_p$ is a vertex (which will be the case in our application towards Theorem~\ref{thm:main}), then $\bigoplus_{p\in P} (X_p,x_p)$ has a structure of a cube complex {whose cubes have the form $\Pi_{p\in P} I_p$, where all but finitely many $I_p$ are equal to $x_p$ and the remaining $I_p$ are edges of $X_p$}. This cube complex is $\CAT(0)$ though this will not be exploited explicitly in the current article.

Let $G_0$ be a group acting on $X_0$ and let $H$ be a group acting on~$P$. The (unrestricted) \emph{wreath product} $G_0\wr_P H$ of $G_0$ and~$H$ over $P$ is the semidirect product of $\Pi_{p\in P} G_p$, where $G_p=G_0$, with the group~$H$ acting on $\Pi_{p\in P} G_p$ by $h\cdot \{g_p\}_{p\in P}=\{g_{h^{-1}(p)}\}_{p\in P}$. 

Let $x_0\in X_0$. We will be considering the subgroup $G^\oplus$ of $G_0\wr_P H$ preserving $\bigoplus_{p\in P} (X_p,x_p)$, where $X_p=X_0$ and $x_p=x_0$, and where the action is defined as follows. For $g=\{g_p\}_{p\in P}\in \Pi_{p\in P}G_p$, we have $g \cdot \{y_p\}_{p\in P}=\{g_p(y_p)\}_{p\in P}$. For $h\in H$, we have $h\cdot \{y_p\}_{p\in P}=\{y_{h^{-1}(p)}\}_{p\in P}$.

It is not hard to verify that if $X_0$ is a simplicial tree, and $f$ is a combinatorial isometry of the cube complex $\bigoplus_{p\in P} (X_p,x_p)$, then $f$ induces a bijection of~$P$. Furthermore, if $G_0$ is the group of all simplicial isometries of $X_0$, then we have $f\in G^\oplus$ if and only if $H$ contains that bijection. For example, if $P=\{p,q\}$ with $X_0$ the real line tiled by unit intervals, then $\bigoplus_{p\in P} (X_p,x_p)$ is the square tiling of the Euclidean plane, and a $90^\circ$ rotation of the plane induces the bijection interchanging $p$ and $q$.

\begin{thm}
\label{thm:maing}
Let $G_0$ be a group acting decently on $X_0$. Let $P$ be the set of {the points of the projective line $\p^1(k)$} over an algebraically closed field $k$, and let $H=\mathrm{Aut}(\p^1)$. Then $G^\oplus$ acts decently on $\bigoplus_{p\in P} (X_p,x_p)$. \end{thm}

We will prove Theorem~\ref{thm:maing} in Section~\ref{sec:main}. {We will deduce Theorem~\ref{thm:main}  from Theorem~\ref{thm:maing} at the end of Section~\ref{sec:Jon}}.
Note that in Theorem~\ref{thm:maing}, we have to make some assumptions on $P$ and $H$. For example, assume that the set $X_0$ is {a finite simplicial tree and $G_0$ is the group of the simplicial isometries of~$X_0$}. Suppose that $H$ is a finitely generated infinite torsion group acting by left multiplication on $P=H$. Then the entire $G^\oplus$ is finitely generated and acts purely elliptically on $\bigoplus_{p\in P} (X_p,x_p)$, but does not have a fixed-point.

\section{Preliminaries}
In this section, we briefly recall some well-known facts about blow-ups and conic bundles. {By Remark \ref{rmk_alg_closed_fields}, it is enough to work over an algebraically closed field for our problem, hence} we assume our base-field $k$ to be algebraically closed. Unless mentioned otherwise, surfaces are assumed to be projective and smooth. 

\subsection{Subgroups of algebraic elements of automorphism groups} 
{Let $S$ be a surface over $k$.} An ample divisor $H$ on $S$ defines a degree function $\deg_H$ on $\Bir(S)$ by $\deg_H(f)=f^*H\cdot H$. An element $g\in \Bir(S)$ is then \emph{algebraic} if the degree sequence $\{\deg_H(f^n)\}_n$ is bounded.  A subgroup $G<\Bir(S)$ is \emph{bounded} if $\{\deg_H(f)\mid f\in G\}$ is bounded. The properties of being algebraic and being bounded do not depend on the choice of an ample divisor (see for example \cite{MR4133708}).

Let us observe that Theorem~\ref{thm:main} holds if we work with finitely generated groups of automorphism groups:

\begin{lemma}\label{lem:aut}
Let $S$ be a 
	surface and let $G<\Aut(S)$ be a finitely generated subgroup such that every element in $G$ is algebraic. Then $G$ is bounded. 
\end{lemma}

\begin{proof}
	The linear action of $\Aut(S)$ on the N\'eron--Severi lattice $N^1(S)$ of $S$ yields a homomorphism $\Aut(S)\to\GL(N^1(S))\simeq \GL_n(\Z)$, whose kernel is an algebraic group and therefore bounded. The image of an element $g\in G$ in $\GL(N^1(S))$ is of finite order, since $g$ is algebraic (see for example \cite[Theorem~4.6]{cantat2015cremona}). Using that $G$ is finitely generated, we obtain that the image of $G$ in $\GL(N^1(S))$  is finite. Therefore, $G$ is a finite extension of a bounded group and therefore bounded itself.
\end{proof}

\subsection{Bubble space, strong factorization, and base-points}
Let $S$ be a surface and let $s\in S$. Then there exists a surface $\tilde S$ and a morphism $\pi\colon\tilde S\to S$ such that the fibre $E$ over $s$ is isomorphic to $\p^1$ and $\pi$ induces an isomorphism between $\tilde S\setminus E$ and $S\setminus\{s\}$. The morphism $\pi\colon \tilde S\to S$ is called the \emph{blow-up} of $S$ in $s$, and it is unique up to isomorphism. Consider two distinct points $s,s'$ on a surface $S$ and their respective blow-ups $\pi_{s}$ and~$\pi_{s'}$. Blowing-up first~$s$ and then~$s'$ gives the same transformation up to isomorphism as blowing-up first~$s'$ and then~$s$. By abuse of language, we will say that $\pi_s$ and~$\pi_{s'}$ \emph{commute} and we denote the successive blow-up of~$s$ and then~$s'$ by $\pi_{s'}\pi_s=\pi_s\pi_{s'}$.

The \emph{bubble space} of $S$ is the set $\B(S)$ of triples $(t,T,\pi)$, where $\pi\colon T\to S$ is a birational morphism from a surface $T$ and $t$ is a point on $T$; two triples $(t,T,\pi)$ and $(t', T',\pi')$ are identified if $\pi^{-1}\pi'$ is a local isomorphism around~$t'$ mapping~$t'$ to~$t$. The bubble space can be thought of as the set of all points on $S$ and on surfaces obtained from~$S$ by successively blowing up  points. The points in $\B(S)$ contained in~$S$ are called \emph{proper} points. 

Zariski's strong factorization theorem states that every birational transformation $f\colon S\dashrightarrow T$ between  surfaces can be factored into blow-ups of  points. More precisely, there exists a surface $Z$ and a factorization
\[
\begin{tikzcd}
	& Z \arrow[ld, "\pi"'] \arrow[rd, "\rho"] &   \\
	S \arrow[rr, "f", dashed] &                                         & T
\end{tikzcd}
\]
where $\pi$ and $\rho$ are compositions of blow-ups of  points. Note that the points blown up by $\pi$ and $\rho$ can be seen as elements of the bubble space $\B( S)$ and $\B( T)$ respectively. Moreover, $Z$ can be chosen in such a way, and will be denoted by $Z_f$, that for any other such factorization
\[
\begin{tikzcd}
	& Z' \arrow[ld, "\pi'"'] \arrow[rd, "\rho'"] &   \\
	S \arrow[rr, "f", dashed] &                                         & T
\end{tikzcd}
\]
there exists a surjective morphism $\eta\colon Z'\to Z_f$ such that $\pi'=\pi\eta$ and $\rho'=\rho\eta$. If we require $Z_f$ to be minimal in this sense, this factorization is unique up to isomorphism and up to possibly changing the order of blowing up the points. {The morphism $\pi\colon Z_f\to S$ is called the \emph{minimal resolution} of $f$}, and the points (in~$\B(S)$) blown up by $\pi$ are called the \emph{base-points} of $f$. 

A base-point $s$ of $f$ is called \emph{persistent} if there exists $l\geq 1$ such that, for all $n\geq l$, the point $s$ is a base-point of $f^n$ but $s$ is not a base-point of $f^{-n}$. 

\subsection{Conic bundles}

A \emph{rational fibration} on a surface $S$ is a morphism $\pi\colon S\to C$, where $C$ is a 
curve, such that all the fibres are isomorphic to $\P^1$. Note that this is a more restrictive definition than the more usual one, which only asks that the fibres are rational curves.

Let $\pi\colon S\to C$ and $\pi'\colon S'\to C$ be rational fibrations. 
We say that a birational transformation $f \colon S\dashrightarrow S'$ \emph{preserves} the fibrations if there exists an automorphism $\hbar(f)\colon C\to C$ such that the following diagram commutes
\begin{center}	
\begin{tikzcd}
	S \arrow[d, "\pi"] \arrow[rr, "f", dashed] &  & S' \arrow[d, "\pi'"] \\
	C \arrow[rr, "\hbar(f)"]                  &  & C.                
\end{tikzcd}
	\end{center}

{ Let $\pi\colon S\to C$ be a rational fibration. An \emph{elementary transformation} $f\colon S\to S'$ is the composition of blowing up a point $s\in S$ followed by blowing down the strict transform of the fibre containing $s$. The surface $S'$ comes equipped with the rational fibration $\pi'\colon S'\to C$ defined by {$\pi=\pi'f^{-1}$}. Note that $f$ preserves these fibrations for $\hbar(f)=\id$.}

The following fact is well-known (see for instance \cite[Corollary~3.2]{schneider2022relations}):
\begin{proposition}\label{prop:elementary}
	Let $\pi\colon S\to C$ and $\pi'\colon S'\to C$ be rational fibrations. 
Every birational transformation $f\colon S\dashrightarrow S'$ preserving the fibrations can be factored { into a sequence of elementary transformations and a (fibration preserving) isomorphism}.
\end{proposition}

A \emph{conic bundle} is a composition with a rational fibration $\pi\colon S\to C$ of a sequence of blow-ups $\tilde S\to S$, such that we blow up {in total} at most one point~$s\in S$ in each fibre. {In other words, $\pi\colon S\to C$ has only finitely many singular fibres, each of them consisting of 
 two rational curves of self-intersection $-1$, called also \emph{$-1$-curves} (the preimage of~$s$ under the blow-up and the strict transform of the fibre of $s$).}

\section{The Jonqui\`eres complex}
\label{sec:Jon}
In this section we reduce Theorem~\ref{thm:main} to Theorem~\ref{thm:maing}. We always assume our base-field $k$ to be algebraically closed, {which is justified by Remark \ref{rmk_alg_closed_fields}, and we assume the} surfaces to be projective and smooth. 

\subsection{The blow-up complex}
Let $T$ be a 
surface. In \cite{lonjou2021actions}, the authors constructed the \emph{blow-up complex} $\CC(T)$ --- a $\CAT(0)$ cube complex with an isometric action of $\Bir(T)$. Let us briefly recall the construction of $\CC(T)$. The vertices of $\CC(T)$ are equivalence classes of \emph{marked surfaces}, i.e., pairs $(S,\varphi)$, where $S$ is a 
surface and $\varphi\colon S\dashrightarrow T$ a birational transformation. Marked surfaces $(S,\varphi)$ and $(S',\varphi')$ are \emph{equivalent} if $\varphi^{-1}\varphi'\colon S'\to S$ is an isomorphism. 
Vertices (represented by) $(S,\varphi)$ and $(\tilde{S},\tilde{\varphi})$, are connected by an edge if ${\varphi}^{-1}\tilde \varphi\colon \tilde S\to {S}$ is the blow-up of a point (or its inverse). More generally, $2^n$ different vertices form an $n$-cube if there is a marked surface $(S,\varphi)$ and distinct points $s_1,\dots, s_n\in S$ such that these vertices are obtained by blowing up subsets of $\{s_i\}_{\{1\leq i \leq n\}}$. 

\subsection{The Jonqui\`eres complex}
Let $\F_0=\p^1\times\p^1$ with the rational fibration $\pi_0\colon \F_0\to\p^1$ onto the first factor. We define the \emph{Jonqui\`eres complex} $\mathscr{J}$ as the subcomplex of the blow-up complex $\mathcal C(\F_0)$ induced on the set of vertices represented by marked surfaces $(S,\varphi)$ such that, for $\pi=\pi_0\varphi$, the rational map
 $\pi\colon S\dashrightarrow\p^1$ is a conic bundle. 

\begin{remark}\label{remark_cube_Jonquièrescplx}
Consider $2^n$ vertices of $\mathscr{J}$ spanning a cube. This means that there exists a surface $(S,\varphi)$ and distinct points $s_1,\dots, s_n\in S$ such that these vertices are obtained by blowing up subsets of $\{s_i\}_{\{1\leq i \leq n\}}$. Since these vertices belong to the Jonquières complex, we have that the points~$s_i$ belong to distinct fibres of $\pi$. In other words, we have $\pi(s_i)\neq\pi(s_j)$ for all $i\neq j$.
\end{remark}

\begin{remark}\label{remark:not_convex}
Note that the Jonquières complex $\Jon$ is not a convex subcomplex of the blow-up complex $\mathcal C(\F_0)$. Indeed, consider the birational transformation $\varphi:(x,y)\mapsto (x,x^2+y)$. It has three base-points, but only one of them is proper. Let $\rho \colon S\to \F_0$ be the minimal resolution of $\varphi$. Then the vertex represented by $(S,\rho)$ is not a vertex of the Jonquières complex because more than one point has been blown-up in the same fibre. However, the vertex represented by $(S,
\rho)$ lies on a geodesic edge-path between the vertices $(\F_0,\id)$ and $(\F_0, \varphi^{-1})$.

Nevertheless, by Proposition~\ref{prop:elementary}, the $1$-skeleton of $\mathscr{J}$ it is isometrically embedded in the $1$-skeleton of $\mathcal C(\F_0)$.
\end{remark}

For any $p\in \p^1$, consider the subcomplex $X_p$ of $\mathscr{J}$ induced on the vertices represented by the marked surfaces $(S,\varphi)$, where $\varphi\colon S\dashrightarrow \F_0$ induces an isomorphism between $S\setminus \pi^{-1}(p)$ and $\F_0\setminus \pi_0^{-1}(p)$. 

\begin{lemma}
\label{lem:tree}
For any $p\in \p^1$, the subcomplex $X_p$ is a tree.
\end{lemma}

\begin{proof}
First note that $X_p$ is connected as a consequence of Proposition \ref{prop:elementary} and the fact that two elementary transformations performed in distinct fibres commute.	

Second, consider the vertex $v_0=(\F_0,\id)$ in $X_p$, and let $\p^1_p\subset \F_0$ be the fibre over~$p$. Let $v_0,v_1,v_2,v_3,\ldots$ be the consecutive vertices on an edge-path without backtracks in $X_p$. The surface $v_1$ is obtained from $v_0$ by blowing up a point {$s$} in $\p^1_p$ to a $-1$-curve. The surface $v_2=(S_2,\varphi_2)$ is obtained from $v_1$ by blowing down the other $-1$-curve in the fibre over $p$ to a point $s_2$. In particular, the birational transformation $\varphi^{-1}_2$ sends the entire {$\p^1_p\setminus \{s\}$} to $s_2$. 

Continuing, for $i\geq 2$ the surface $v_{2i-1}$ is obtained from $v_{2i-2}$ by blowing up a point in the fibre over $p$ distinct from $s_{2i-2}$ to a $-1$-curve, and the surface $v_{2i}$ is obtained from $v_{2i-1}$ by blowing down the other $-1$-curve in the fibre over $p$ to a point $s_{2i}$. Thus the birational transformation $\varphi^{-1}_{2i}$ sends the entire {$\p^1_p\setminus \{s\}$} to $s_{2i}$. In particular, $\varphi_{2i}$ {is not an isomorphism}, and so $v_{2i}\neq v_0$. This proves that there is no cycle in $X_p$, and so $X_p$ is a tree.
\end{proof}

We choose a preferred vertex
$x_p=(\F_0,\id)$ in each $X_p$, and we consider the family of pointed sets $\{(X_p,x_p)\}_{p\in \p^1}$.

\begin{lemma}\label{lem_product_tress}
	The cube complex $\mathscr{J}$ is isomorphic to $\bigoplus_{p\in \p^1} (X_p,x_p)$.
\end{lemma}

\begin{proof}[Proof of Lemma~\ref{lem_product_tress}]
Let $(S,\varphi)$ be 
a vertex of $\mathscr{J}$. The birational transformation $\varphi \colon S\dashrightarrow \F_0$ is decomposed as $\varphi=\sigma_n\cdots\sigma_1$, where each $\sigma_i$ is the blow-up of a point or the blow-down of a $-1$-curve in the fibre over a point $p_i\in \p^1$. For $p_i\neq p_j$, 
we have that $\sigma_i$ and $\sigma_j$ commute. Thus there is a finite subset $Q\subset \p^1$ such that 
$\varphi=\Pi_{p\in Q}\varphi_p$, where each (of the commuting) $\varphi_p$ is a composition of~$\sigma_i$ with $p_i=p$. 
Consider the marked surfaces $\varphi_p\colon S_p \dashrightarrow\F_0$ for $p\in Q$, and $\varphi_p=\id, S_p=\F_0$ for $p\in \p^1\setminus Q$. They represent vertices of~$X_p$. Then $\{(S_p,\varphi_p)\}_{p\in \p^1}$
defines a vertex of $\bigoplus_{p\in \p^1} (X_p,x_p)$. 

This is a bijective correspondence between the vertices of $\mathscr{J}$ and $\bigoplus_{p\in \p^1} (X_p,x_p)$, and it extends to an isomorphism on the entire complexes.
\end{proof}

The coordinate $y_p$ of a point $y\in  \bigoplus_{p\in \p^1} (X_p,x_p)=\mathscr{J}$ should be thought of as the ``marked fibre''
in the surface corresponding to $y$ over the point $p$. Thus modifying the coordinate $y_p$ of $y$ corresponds to performing an alternating sequence of blow-ups of points and blow-downs of $-1$ curves in the fibre over $p$ of the surface corresponding to $y$.

\smallskip

	The \emph{Jonqui\`eres group} is the subgroup of $\Bir(\F_0)$ consisting of the \emph{Jonqui\`eres transformations} $f$ that preserve the rational fibration $\pi_0\colon \F_0\to \p^1$. {The Jonqui\`eres group acts on the vertex set of $\Jon$ by $f\cdot(S, \varphi)=(S, f\varphi)$,
and this action extends to an action by isometries on the entire $\Jon$.}
	
	\begin{remark}\label{rmk_subjonq_fixpoint}
	A subgroup $G$ of the Jonqui\`eres group is a subgroup of $\Aut(S)$ for some 
conic bundle $\pi \colon S\to \p^1$ if and only if $G$ fixes a point in $\Jon$. This is because if $G$ fixes an interior point of a cube in~$\mathscr{J}$ described via a surface $S$ in Remark~\ref{remark_cube_Jonquièrescplx}, then $G$ {fixes this cube and in particular $G$} fixes the vertex corresponding to~$S$. {By the definition of a marked surface this is equivalent to} $G< \Aut(S)$.
	\end{remark}

{We conclude with the following.}

\begin{proof}[Proof of Theorem~\ref{thm:main}]
By Remark \ref{rmk_alg_closed_fields}, we can assume the {base-field $k$} to be algebraically closed. By Remark \ref{rmk_redu_Cantat}, we can assume $G$ to be a subgroup of the Jonqui\`eres group.  

As a consequence of Lemma \ref{lem_product_tress}, the complex $\bigoplus_{p\in \p^1} (X_p,x_p)$ inherits the action of the Jonqui\`eres group from $\Jon$. Because $X_p$ are trees (see Lemma~\ref{lem:tree}), {their isometry groups are decent}. All the elements of the Jonqui\`eres group induce elements of $\mathrm{Aut}(\p^1)$ on $\p^1$. 
	Moreover, by Remark \ref{rmk_subjonq_fixpoint} and by Lemma~\ref{lem:aut}, to prove Theorem~\ref{thm:main} for a subgroup $G$ of the Jonqui\`eres group, we need to find a fixed-point for $G$ in $\Jon$, {which is guaranteed by} Theorem~\ref{thm:maing}. 
\end{proof}

\section{Proof of the main theorem}
\label{sec:main}

In this section we prove Theorem~\ref{thm:maing}. We keep the notation $\Jon$ for the general  $\bigoplus_{p\in \p^1} (X_p,x_p)$ (and not just for the Jonqui\`eres complex). We denote by $\hbar$ the quotient map $G_0\wr_P H\to H$. Abusing the notation, for $f\in G_0\wr_P H$ and $p\in P$, we denote by $f(p)\in P$ the point $\hbar(f)(p)$.

\begin{conv}
	\label{conv:fibre}
	Note that for $f=hg\in G^\oplus$, with $h\in H$ and $g\in\Pi_{p\in P} G_p$, and for $y\in \Jon$, the value $f(y)_p$ equals $g_{h^{-1}(p)}\big(y_{h^{-1}(p)}\big)$ and so it depends only on $h,g_{h^{-1}(p)}$ and $y_{h^{-1}(p)}$. Henceforth, slightly abusing the notation, we will refer to $f(y)_p$ as $f(y_{f^{-1}(p)})$.
\end{conv}

This conveys the fact that for a Jonqui\`eres transformation $f$, the ``marked fibre'' in $f(S)$ over $p$ depends only on $f$ and the ``marked fibre'' in $S$ over $f^{-1}(p)$.

{
\begin{lemma}
\label{lem:first_item}
Let $G_0$ be a group acting on $X_0$ so that each subgroup of $G_0$ with a finite orbit fixes a point of $X_0$ (which is the first item of Definition~\ref{def:decent} for $G_0$). Then each subgroup of $G<G^\oplus$ with a finite orbit $Y$ in $\Jon$ fixes a point of $\Jon$ (which is the first item of Definition~\ref{def:decent} for $G$).
\end{lemma}}

\begin{proof}
{
For $p\in P$, let $Y_p=\{y_p \colon y\in Y\}$ denote the finite set of the coordinates of~$Y$ in the factor $X_p$.}

{First, we consider each finite orbit $O\subset P$ of $\hbar(G)<H$. If for all $p\in O$ we have $Y_p=\{x_p\}$, then we define $z_p=x_p$ for all $p\in O$. 
Since $Y\subset \Jon$, there are only finitely many finite orbits $O$ of $\hbar(G)<H$ with $Y_p\neq \{x_p\}$ for some $p\in O$. For each such $O$, we choose $o\in O$. Let $G'_o$ be the projection to $G_o$ of the stabiliser of $o$ in~$G$, that is, the group of all $g_o$ over $g\in G$ fixing $o$. Note that $Y_o$ contains an orbit of~$G_o'$. Thus $G_o'$ has a finite orbit and hence a fixed-point $z_o\in X_o$ by the hypothesis of the lemma. We set $z_p=f(z_o)$ (see Convention \ref{conv:fibre}) for any $f\in G$ with $f(o)=p$, which only depends on $p$ and not on $f$ since $z_o$ was fixed by $G_o'$. }

{Now, let $O\subset P$ be an infinite orbit of $\hbar(G)<H$. Since $Y\subset \Jon$, there is $o\in O$ with $Y_o=\{x_o\}$. Consequently, for any $p\in O$, we have that $Y_p$ consists only of a single element, which we call $z_p$.
Since $Y\subset \Jon$, this $z_p$ equals $x_p$ for all but finitely many $p$.}

{Consequently, we have $z=\{z_p\}_p\in \Jon$. By construction, $z$ is a fixed-point of~$G$.}
\end{proof}

\subsection{Biregularity}

The following encapsulates the idea of a Jonqui\`eres transformation $f$ having a persistent base-point in the fibre over a point $p\in P$.

\begin{defin}
\label{def:pers} Let $z=\{z_r\}_r\in \Jon$ be a distinguished vertex. Let $p\in P$ and $f\in G^\oplus$. We say that $f$ is \emph{biregular over $p$} (with respect to $z$) if $f(z)_{f(p)}=z_{f(p)}$ (or, in our notation from Convention~\ref{conv:fibre}, $f(z_p)=z_{f(p)}$). Equivalently, for $f=hg$ with $h\in H, g\in \Pi_pG_0$, we have $g_p(z_p)=z_{f(p)}$ (which equals $z_{h(p)}$). Otherwise, we say that $f$ is \emph{singular} (w.r.t.\ $z$) \emph{over} $p$.

An element $f\in G^\oplus$ \emph{has persistent fibre over $p\in P$} if there exists $l\geq 1$ such that, for all $n\geq l$, we have that $f^n$ is singular over $p$ and $f^{-n}$ is biregular over $p$. 
\end{defin}

\begin{rem}
\label{rem:regular}
Note that $f$ is biregular over $p$ if and only if $f^{-1}$ is biregular over~$f(p)$. Furthermore, if $f$ is biregular over $p$ and $f'$ is biregular over $f(p)$, then $f'f$ is biregular over $p$.
\end{rem}

\begin{rem}
\label{rem:fixed}
The point $z$ is a fixed-point for $f$ if and only if $f$ is biregular over all $p\in P$.
\end{rem}

\begin{rem}
\label{rem:per}
If $f$ has persistent fibre over $p$, then it does not have a fixed-point $y\in \Jon$. Indeed, the orbit of $p$ under~$\langle f\rangle$ is infinite, since $f^m(p)=p$ with $m>0$ would imply, by Remark~\ref{rem:regular}, that $f^{lm}$ is simultaneously biregular and singular over~$p$. Furthermore, for all $n\geq l$, we have $z_{f^{n}(p)}\neq f^{n}(z_p)=f^n\big(f^n(z_{f^{-n}(p)})\big)$.
For large~$n$, we have $y_{f^{\pm n}(p)}=z_{f^{\pm n}(p)}$, contradicting $y_{f^{n}(p)}=f^{2n}(y_{f^{-n}(p)})$.
\end{rem}

\subsection{Abelian case}

The following proves Theorem~\ref{thm:maing} in the case where $\hbar(G)$ is abelian.  Note that here, as well as in Lemma~\ref{lem:semisimple}, we do not need to assume that $P$ is the projective line $\p^1$.

\begin{lemma}
\label{lem:abel}
Let $G_0$ be a group acting decently on $X_0$. Let $H$ be a group acting on $P$, and let $G<G^\oplus$ be such that either $\hbar(G)<H$ is trivial or it contains an element that has only finitely many finite orbits on $P$. If $G$ is finitely generated, acts purely elliptically on $\Jon$, and $\hbar(G)$ is abelian, then $G$ fixes a point of $\Jon$.
\end{lemma}

\begin{proof} If $\hbar(G)$ is trivial, then $G<\Pi_{p\in P} G_p$. Let $Q\subset P$ be a finite set such that each generator $\{g_p\}_{p\in P}$ of $G$ satisfies $g_p(x_p)=x_p$ for $p\notin Q$. For each $q\in Q$, the projection of $G$ to $G_q$ is a purely elliptic subgroup, so it fixes a point $y_q\in X_q$. Setting $y_p=x_p$ for $p\notin Q$, we obtain a fixed-point $\{y_p\}_p$ for $G$ in $\Jon$.

Thus from now on we can assume that there is $t\in G$ with $\hbar(t)$ having only finitely many finite orbits on $P$. Let $z\in \Jon$ be a fixed-point of $t$. Then $t$ is biregular w.r.t.\ $z$ over each $p\in P$ {by Remark~\ref{rem:fixed}}. 
Let $Q\subset P$ be the union of the finite orbits of~$\langle t\rangle$. Note that $Q$ is preserved by $\hbar(G)$, since $\hbar(G)$ is abelian. 

Let $p\in P\setminus Q$. We claim that any $f\in G$ is biregular w.r.t.\ $z$ over $p$.
To justify the claim, let $B$ be the finite set of $b\in P\setminus Q$ over which $f$ or $f^{-1}$ is singular. We will show that for some $n>0$, setting $f_n=t^nf$, we have $f_n^i(p)\notin B$, for all $i \in \Z\setminus \{0\}$.
To find such $n$, suppose first that $f^l(p)=t^k(p)$ for some $k,l\in \Z$ with $l> 0$. Then choose $m_0>0$ such that, for all $m\geq m_0$, and all $0\leq j<l$, we have $t^{\pm m}f^j(p)\notin B$. It suffices then to take $n=m_0+|k|$,
since for any $i\neq 0$ we have $f_n^i(p)=t^{ni}f^i(p)=t^{\pm m}f^j(p)$ for some $0\leq j<l$ and $m\geq m_0$. Otherwise, if there is no $l\neq 0$ with $f^l(p)\in \langle t \rangle(p)$, then
there are finitely many $k\in \Z$ with $t^k(p)\in \langle f\rangle B$, since $B$ is finite. It suffices then to take $n$ larger than the maximum of their $|k|$.

Thus we have $f_n^i(p)\notin B$, for all $i \in \Z\setminus \{0\}$. If $f$ was singular w.r.t.\ $z$ over $p$, then by Remark~\ref{rem:regular} $f_n\in G$ or its inverse would have persistent fibre over $p$ (with $l=1$), contradicting Remark~\ref{rem:per} and justifying the claim. 

For each orbit $O$ of $\hbar(G)$ in $Q$, choose $o\in O$ and $y_o\in X_o$ that is fixed by the projection to $G_o$ of the stabiliser of $o$ in $G$, which is of finite index in $G$, hence finitely generated. For each $f\in G$, choose $y_{f(o)}=f(y_o)$ (notation from Convention~\ref{conv:fibre}),
which only depends on $f(o)$ and not on $f$, since $y_o$ was fixed by the projection to $G_o$ of the stabiliser of $o$ in~$G$. 
Setting $y_p=z_p$ for $p\notin Q$ gives us a fixed-point $y$ for~$G$.
\end{proof}

\subsection{Semisimple case}

In order to treat the case of non-abelian $\hbar(G)$, we use the dynamics of elements in $\Aut(\p^1)=\PGL_2(k)$. 
An element $a\in\PGL_2(k)$ is \emph{semisimple}, if it is conjugate to a diagonal element. With respect to suitable local coordinates, the automorphism of $\p^1$ induced by $a$ is given by $z\mapsto \lambda z$ for some $\lambda\in k^*$. Let us observe that $a$ is of infinite order if and only if $\lambda$ is not a root of unity. In this case, $a$ fixes exactly two points of $\p^1$ and it does not have any other finite orbit. If $k$ admits a norm $|\cdot|$ with $|\lambda|\neq 1$, then $a$ has north-south dynamics, defined below, in the topology of $\p^1$ induced by~$|\cdot|$. 

\begin{defin} 
\label{def:NS}
Let $a$ be a homeomorphism of a Hausdorff topological space $P$, fixing $p,q\in P$, and having the following property. For any disjoint open sets $U\ni p,V\ni q$, there is~$n$ such that $a^n(P\setminus V)\subset U$ and  $a^{-n}(P\setminus U)\subset V$.
We then say that $a$ has \emph{north-south dynamics}.
\end{defin}

\begin{lemma}
\label{lem:semisimple}
Let $G_0$ be a group acting decently on $X_0$. Let $H$ be a group acting on $P$. Let $a\in H$ have north-south dynamics {with $p,q$ as in Definition~\ref{def:NS}} for some Hausdorff topology on $P$ in which $H$ acts by homeomorphisms. Suppose also that $\mathrm{Stab}(p)\cap\mathrm{Stab}(q)<H$ is abelian.
If $G<G^\oplus$ is finitely generated, acts purely elliptically on $\Jon$, and $\hbar(G)$ contains $a$, then $G$ has a fixed-point in $\Jon$.
\end{lemma}
\begin{proof}
Choose $t\in G$ with $\hbar(t)=a$. Let $z\in \Jon$ be a fixed-point of $t$. We claim that for any $f\in G$ and any $r\in P\setminus \{p,q\}$ with $f(r)\neq p,q$, we have that $f$ is biregular w.r.t. $z$ over $r$. 

To justify the claim, first consider the case where $\hbar(f)$ interchanges $p$ and $q$. Then the group $\langle a, \hbar(f)\rangle$ is virtually abelian, {since it has an index $2$ subgroup contained in $\mathrm{Stab}(p)\cap\mathrm{Stab}(q)$}. By Lemma~\ref{lem:abel}, the group $\langle t,f\rangle$ has a finite orbit, hence a fixed-point $y$ in $\Jon$ by Lemma~\ref{lem:first_item}. {Since both $y$ and $z$ belong to $\Jon$}, all but finitely many coordinates of $y$ have to be that of $z$. {Thus} $t(y)=y$ implies $y_r=z_r$ for all $r$ in all infinite orbits of~$\langle t\rangle$, so for all $r\neq p,q$. Since $f$ fixes $y$, we have that $f$ is biregular w.r.t.\ $z$ over $r$, as desired.

Second, consider the case where $\hbar(f)$ does not interchange $p$ and $q$. Then after possibly replacing $t$ with $t^{-1}$ and interchanging $p$ with $q$, we can assume $f(p)\neq q$. Let $B\subset P$ be the finite set of points over which $f$ or $f^{-1}$ is singular w.r.t.\ $z$. Let $U\ni p$ (respectively, $V\ni q$) be an open neighbourhood intersecting $B\cup \{f^{-1}(p),f(q)\}$ only possibly at $p$ (respectively, $q$), and such that $f(U)$ is disjoint from $V$, which is possible since $f(p)\neq q$.
{Assume also that $U$ is disjoint from $V$.}
Let $n>0$ be as in Definition~\ref{def:NS}. 

If $f$ is singular over~$r$, then we have $r,f(r)\in P\setminus V$. Thus $t^nf(r)\subset U$, but $t^nf(r)\neq p$ since $f(r)\neq p$.
Since $f(U)$ is disjoint from $V$, and $U$ does not contain $f^{-1}(p)$, unless $f(p)=p$, we obtain inductively, for all $m>0$, that $(t^nf)^m(r)\in U\setminus p$. We also have $t^{-n}(r)\in V\setminus q$ and analogously we obtain $t^{-n}(t^nf)^m(r)\in V\setminus q$ for all $m<0$. Consequently, $t^nf$ has persistent fibre over $r$ (with $l=1$), which contradicts Remark~\ref{rem:per} and finishes the proof of the claim.

Note that if the entire $\hbar(G)$ stabilises $\{p,q\}$, then we are done by Lemma~\ref{lem:abel}. Suppose also for the moment that $\hbar(G)$ does not {fix} $p$ or $q$. Then there is $f\in G$ and $r\neq p,q$ with $f(r)=p$. Consider another $f'\in G$ with $r'\neq p,q$ and $f'(r')=p$. Then $f^{-1}f'(r')=r.$ Applying the claim above to $f^{-1}f'$, we have $f^{-1}f'(z)_{r'}= z_{r}$. Consequently, $f'(z)_{p}= f(z)_{p}$. We now replace the coordinate $z_p$ of $z$ by $f(z)_{p}$, which, as we have seen, does not depend on~$f$. Note that the new $z$ is still fixed by $t$, which can be verified by substituting above $f'=tf$. Furthermore, now $f$ is biregular over $r$ and $f^{-1}$ is biregular over $p$.
We analogously change the $z_q$ coordinate of~$z$. 

We will verify that the new $z$ is a fixed-point for $G$. It remains to verify the biregularity of $f \in G$ over $u\in \{p,q\}$ in the case where $f(u)\in \{p,q\}$. Choose any $r\neq p,q$ and $f'\in G$ with $f'(r)=u$. Introduce $f''=ff'$, which satisfies $f''(r)=f(u)$. From the previous paragraph it follows that both $f',f''$ are biregular over $r$. This implies that $f$ is biregular over $u$, as desired.

In the case where $\hbar(G)$ {fixes}, say, $p$, we redefine $z_p$ to be the fixed-point of the projection of $G$ to $G_p$.
\end{proof}

\subsection{Conclusion}

Recall that an element $h\in\PGL_2(k)$ is \emph{unipotent}, if it is conjugate to an element of the form $\left(\begin{array}{cc}
	1 & c \\
	0 & 1
\end{array} \right)$ for some $c\in k$. By considering the Jordan decomposition, we observe that, for $k$ algebraically closed, every element of $\PGL_2(k)$ is either unipotent or semisimple.

\begin{proof}[Proof of Theorem~\ref{thm:maing}]
Suppose first that $\hbar(G)$ contains a semisimple element $a$ of infinite order with eigenvalues $\lambda, \lambda^{-1}$ and fixed-points $p,q\in \p^1$. Since $a$ is of infinite order, $\lambda$ is not a root of unity.  Let $z\in \Jon$. Let $\tilde{k}\subset k$ be the smallest field such that the points $p$ and $q$, the scalar $\lambda$, the elements of $\hbar(G)$, and all the points over which the elements of $G$ are singular w.r.t.\ $z$, are defined over $\tilde{k}$. Since $G$ is finitely generated, the field $\tilde{k}$ is a finitely generated field extension over the prime field of~$k$.

Let $\Jon_{\tilde{k}}=\bigoplus_{p\in \p^1(\tilde{k})} (X_p,x_p)$, where $\p^1(\tilde{k})$ is the set of the $\tilde k$-rational points of~$\p^1$. Note that the action of $G$ on $\Jon$ projects to an action on $\Jon_{\tilde{k}}$. 

By \cite[Theorem 2.65]{drutukapovich}, we can embed $\tilde{k}$ as a subfield into some local field $K$ with norm $|\cdot|$ such that $|\lambda|\neq 1$. Then $a$ has north-south dynamics on $\p^1(\tilde{k})$ with respect to the Hausdorff topology on $\p^1(\tilde{k})$ induced by the topology of $\p^1(K)$. By Lemma~\ref{lem:semisimple} applied with $P=\p^1(\tilde{k})$, we have that $G$ fixes a point $\{y_p\}_{p\in\p^1(\tilde{k})}\in\Jon_{\tilde{k}}$. Since all the elements of $G$ are biregular w.r.t.\ $z$ over all the points outside $\p^1(\tilde{k})$, we obtain that $G$ fixes the point $\{y_p\}_{p\in \p^1}\in\Jon$, where $y_p=z_p$  for $p\notin\p^1(\tilde{k})$.

Otherwise, if $\hbar(G)$ does not contain a semisimple element of infinite order, then all the elements of $\hbar(G)$ are unipotent or of finite order. By \cite[Proposition~14.46]{drutukapovich}, there is a finite index subgroup $G'<G$ with $\hbar(G')$ conjugate into the abelian subgroup of the elements of the form $\left(\begin{array}{cc}
	1 & c \\
	0 & 1
\end{array} \right).$
Lemma~\ref{lem:abel} implies that also in this case $G'$ (and hence $G$) fixes a point of $\Jon$.
\end{proof}

{\section{Proofs of Theorem~\ref{thm:main_general} and Corollary~\ref{cor:degree}}}

\subsection{Non-rational surfaces}\label{sec:remarks} 
Here, we give a proof of Theorem~\ref{thm:main_general}. 

\begin{proof}[Proof of Theorem \ref{thm:main_general}]
{Again, by Remark~\ref{rmk_alg_closed_fields}, we can assume that $k$ is algebraically closed}.
	If $S$ is rational, then the result follows from Theorem~\ref{thm:main}. 
	
	If the Kodaira dimension of $S$ is non-negative, then there exists a smooth projective surface $T$ birationally equivalent to $S$ such that $\Bir(T)=\Aut(T)$ and the result follows from Lemma~\ref{lem:aut}.
	
	Finally, if the Kodaira dimension of $S$ is $-\infty$, but $S$ is not rational, then $S$ is birationally equivalent to $\p^1\times C$ for some non-rational smooth curve $C$. In this case, all the elements  $f\in\Bir(S)$ preserve the rational fibration given by the projection to~$C$.  {Indeed, if $F$ is a general fibre of the second projection $\pi_2\colon S\to C$, then the restriction of $\pi_2\circ f$ to $F$ induces a rational map to $C$. Since $F\cong \p^1$, this rational map cannot be dominant, hence its image is a point, which implies that $f(F)$ is another fibre.} If the genus of $C$ is $>1$, then $\Bir(C)=\Aut(C)$ is finite. If the genus of~$C$ is~$1$, then $C$ is an elliptic curve and $\Bir(C)=\Aut(C)$ is virtually abelian. We can thus apply Lemma~\ref{lem:abel}.
\end{proof}

\subsection{Degree growth of finitely generated groups}\label{sec:degrees}

There is a well-known and important correspondence between the dynamical behaviour of birational transformations in $\Cr_2(k)$ and the type of isometry they induce on the infinite dimensional hyperbolic space $\h$. The following theorem gives in particular a precise description of the degree growth. It is due to several people. We refer to \cite{cantat2015cremona} for details and references. 

\begin{theorem}[Gizatullin; Diller and Favre; Cantat]\label{thm:hyperboloid}
	Let $\kk$ be an algebraically closed field and $f\in\Cr_2(k)$. Then one of the following is true:
	\begin{enumerate}
		\item The transformation $f$ is algebraic, the isometry of $\h$ induced by $f$ is elliptic, and the degree sequence $\{\deg(f^n)\}$ is bounded.	
		\item[(2a)] The isometry of $\h$ induced by $f$ is parabolic, $\deg(f^n)\sim cn$ for some $c>0$, and $f$ preserves a rational fibration. 
		\item[(2b)] The isometry of $\h$ induced by $f$ is parabolic, $\deg(f^n)\sim cn^2$ for some $n$, and $f$ preserves a fibration of genus 1 curves. 
		\item[(3)] The isometry of $\h$ induced by $f$ is loxodromic, $\deg(f^n)=c\lambda^n +\mathcal{O}(1)$ for some $c>0$ and $\lambda>1$. 
	\end{enumerate}
\end{theorem}

We prove now Corollary~\ref{cor:degree}:

\begin{proof}[Proof of Corollary~\ref{cor:degree}]
	If all elements in $G$ induce elliptic isometries on $\h$, i.e., they are algebraic, then $G$ and hence  $D_T(n)$ are bounded by Theorem~\ref{thm:main}. This implies that the orbit of $G$ on $\h$ is bounded and hence that $G$ fixes a point in~$\h$.
	
	Next, consider the case, where no element in $G$ induces a loxodromic isometry of $\h$, but there is an element $f\in G$ inducing a parabolic isometry. 
	
	First, assume that $f$ preserves a rational fibration. Then all elements in $G$ preserve this same rational fibration (see for instance \cite[Lemma~5.3.4]{urech2017subgroups}) and after conjugation we may assume that $G$ is a subgroup of the Jonqui\`eres group (note that the asymptotic growth of $D_T(n)$ is invariant under conjugation). For an element $g$ in the Jonqui\`eres group we have $\deg(g)=\frac{\#\bp(g)+1}{2}$ (see for instance \cite{Lamy_Cremona}), where $\#\bp(g)$ denotes the number of base-points of $g$. Since $\#\bp(gh)\leq \#\bp(g)+\#\bp(h)$ we obtain that that $D_T(n)\leq Kn$, where $K=\max_{g\in T}\{\#\bp(g)\}$. At the same time, since by assumption $G$ contains an element whose degree growth is linear, we have $kn\leq D_T(n)$, for some $k>0$. Hence, $D_T(n)\asymp n$.
	
	Now, assume that $f$ preserves a fibration of curves of genus $1$. Again, this implies that all of $G$ preserves the same fibration of curves of genus $1$ and after conjugation we may assume that $G$ is a subgroup of $\Aut(S)$, where $S$ is a Halphen surface \cite[Lemma~5.3.4]{urech2017subgroups}. In this case, the statement follows from Lemma~\ref{lem:halphen} below.
	
	Finally, we consider the case, where $G$ contains an element $f$ inducing a loxodromic isometry on $\h$. By Theorem~\ref{thm:hyperboloid}, there exist $c$ and $\lambda$ such that $\deg(f^n)=c\lambda^n+\mathcal{O}(1)$. On the other hand, for $\lambda_2=\max_{g\in T}\{\deg(g)\}$ we have $D_T(n)\leq \lambda_2^n$. This shows that $f\asymp \lambda^n$.
\end{proof}

A \emph{Halphen surface} is a rational smooth projective surface $S$ such that  $|−mK_S|$ is a pencil of genus 1 curves with empty base locus for some $m>0$. The ideas used in the following lemma have been described in \cite{gizatullin1980rational} (see also \cite{grivaux2013parabolic} and \cite{cantat2012rational}). We follow the account described in \cite{Lamy_Cremona}.

\begin{lemma}\label{lem:halphen}
	Let $S$ be a Halphen surface and let $G<\Aut(S)$ be a finitely generated subgroup containing a non-algebraic element $f$. Then $D_T(n)\asymp cn^2$ for some $c>0$. 
\end{lemma}

\begin{proof}
	Let us first recall that after possibly passing to a finite index subgroup  (which does not change the asymptotic growth of $D_T(n)$), we may assume that $G$ is abelian. Moreover, all algebraic alements in $\Aut(S)$ are of finite order and $\Aut(S)$ does not contain any element inducing a loxodromic isometry on $\h$  (see for instance \cite{cantat2015cremona} for these facts). Hence, again up to passing to a finite index subgroup, we may assume that all elements in $G$ induce a parabolic isometry on~$\h$.

	Let $N_{\mathbb{R}}(S)$ be the N\'eron--Severi space of $S$. Recall that $N_{\mathbb{R}}(S)$ comes with an intersection form of signature $(1,\dim(N_{\mathbb{R}}(S))-1)$, which is preserved by the action of $\Aut(S)$ by push-forwards. There exists a nef divisor class $D_0\in N_{\mathbb{R}}(S)$ such that $D_0\cdot D_0=0$ and such that $g_*D_0=D_0$ for all $g\in G$ (in fact, we can take $D_0=mK_S$). The assumption that all elements in $G$ induce parabolic isometries on $\h$ implies that for all $f\in G$, the only eigenvectors of $f$ are multiples of $D_0$.  For all $g\in G$, the restriction of $g_*$ to $D_0^\perp/D_0$ has finite order, since $g_*$ preserves an integral lattice and the induced intersection form on $D_0^\perp/D_0$ is negative definite. Hence, up to passing to a finite index subgroup of $G$, we may assume that the restriction of $G$ to $D_0^\perp/D_0$ is the identity. Let $f_1,\dots, f_k$ be generators of $G$.
	
	Let $A\in N_{\mathbb{R}}(S)$ be an ample divisor. Note that we have $f_*A\neq A$ for all $f\in G$ and $A\notin D_0^\perp$. Write $(f_i)_*A=A+R_i$, where $R_i\in D_0^\perp$. Since the restriction of $(f_i)_*$ to $D_0^\perp/D_0$ is the identity, we can write $(f_i)_*R_j=R_j+t_{ij}D_0$ for some $t_{ij}\in\R$. Let us observe that $(f_i)_*(f_j)_*A=(f_j)_*(f_i)_*A$ for all $i$ and $j$ implies that $t_{ij}=t_{ji}$ for all $i$ and $j$. 
	
	By induction, we obtain
	\[
	(f_i)^n_*A=A+nR_i+\frac{n(n-1)t_{ii}}{2}D_0,
	\]
	and, as a consequence,
	\[
	(f_1)^{n_1}_*\cdots (f_k)^{n_k}_*A=A+\sum_in_iR_i+\left(\sum_i\frac{n_i(n_i-1)}{2}t_{ii}+\sum_{i<j}n_in_jt_{ij}\right)D_0.
	\]

Since, by assumption, all the $f_i$ are non-algebraic and preserving a fibration of genus 1 curves, the sequence $\deg(f^n)$ grows quadratically in $n$, by Theorem~\ref{thm:hyperboloid}, and hence the $t_{ii}$ are positive. We obtain that 
	
	\[
	\deg_A(f_1^{n_1}\cdots f_k^{n_k})=((f_1)^{n_1}_*\cdots (f_k)^{n_k}_*A)\cdot A.
	\]
	Since $A$ is ample, we have that $A\cdot D_0>0$. Hence, $D_T(n)$ has quadratic growth.
\end{proof}

\section*{Acknowledgments}
We thank Serge Cantat for his valuable comments that helped to improve the exposition of this article.  {We also thank the anonymous referees for their helpful suggestions.}
The first author would like to thank the CRM (Centre de Recherche Mathématiques de Montreal), the Simons foundation and the organisers of the thematic semester ``Théorie géométrique des groupes'' for her stay at the CRM where a part of this project was realized.

\bibliographystyle{amsalpha}
\bibliography{bibliography_cu}

\end{document}